\documentclass[a4paper,11pt]{amsart}
\usepackage{graphicx}
\usepackage{amssymb}
\usepackage[mathscr]{euscript}
\usepackage{tikz}
\usetikzlibrary{matrix,arrows,decorations.pathmorphing,decorations.markings,calc}
\usepackage{tikz-cd}
\tikzset{->-/.style={decoration={
  markings,
  mark=at position #1 with {\arrow{>}}},postaction={decorate}}}

\theoremstyle{plain}
  \newtheorem{thm}{Theorem}[section]
  \newtheorem{defn}{Definition}[section]
  \newtheorem{prop}{Proposition}[section]

\theoremstyle{definition}
  \newtheorem{example}{Example}[section]
\theoremstyle{remark}
  \newtheorem*{rem}{Remark}

\newcommand{\mf}{\mathfrak}
\newcommand{\on}{\operatorname}

\newcommand{\g}{\mathfrak{g}}
\newcommand{\h}{\mathfrak{h}}

\newcommand{\Hom}{\operatorname{Hom}}
\newcommand{\la}{\langle}
\newcommand{\ra}{\rangle}
\newcommand{\R}{\mathbb{R}}
\newcommand{\half}{\frac{1}{2}}

\newcommand{\gqP}{$\g$-quasi-Poisson}
\newcommand{\Lgqb}{Lie $\g$-quasi-bialgebra}

\title{Lie groups in quasi-Poisson geometry and braided Hopf algebras}
\author{Pavol \v{S}evera}
\address{Section of Mathematics, University of Geneva, Switzerland}
\email{pavol.severa@gmail.com}
\author{Fridrich Valach}
\address{Section of Mathematics, University of Geneva, Switzerland}
\email{fridrich.valach@gmail.com}
\thanks{Supported in part by  the grant MODFLAT of the European Research Council and the NCCR SwissMAP of the Swiss National Science Foundation.}
\begin{document}

\begin{abstract}
We extend the notion of Poisson-Lie groups and Lie bialgebras from Poisson to $\g$-quasi-Poisson geometry and provide a quantization to braided Hopf algebras in the corresponding Drinfeld category. The basic examples of these $\g$-quasi-Poisson Lie groups are  nilpotent radicals of parabolic subgroups. We also provide examples of moment maps in this new context coming from moduli spaces of flat connections on surfaces.
\end{abstract}
\maketitle

\section{Introduction}

Quasi-Poisson geometry, introduced by Alekseev, Kosmann-Schwarzbach and Meinrenken \cite{AKM} in their study of moduli spaces of flat connections, is an ``infinitesimally braided'' version of Poisson geometry. A $\g$-quasi-Poisson manifold is by definition a manifold with an action of a Lie algebra $\g$ and with an invariant bivector field $\pi$, satisfying
$$[\pi,\pi]/2=\rho(\phi),$$
where $\rho$ is the action of $\g$ and $\phi\in\bigwedge^3\g$ is built from an invariant inner product on $\g$ and from the structure constants of $\g$.

The inner product on $\g$ (and a Drinfeld associator) was used by Drinfeld \cite{D} to deform the symmetric monoidal structure on the category of $U\g$-modules to a braided monoidal structure. As observed by Enriquez and Etingof \cite{EE}, the natural definition of deformation quantization of a $\g$-quasi-Poisson manifold is a deformed product which is associative in Drinfeld's category.

In this work we extend the theory of Poisson-Lie groups and Lie bialgebras to the quasi-Poisson world, and we also quantize these quasi-Poisson versions to braided Hopf algebras in the Drinfeld category. This means, in particular, that the deformed product and coproduct satisfy the relation
$$
\begin{tikzpicture}[scale=0.8]
\draw(0,0)--(1,1)--(2,0) (1,1)--(1,2) (0,3)--(1,2)--(2,3);
\node at (3,1.5) {=};
\begin{scope}[xshift=4cm]
\draw (2,0.7)--(0,2.3);
\draw[white,line width=5] (0,0.7)--(2,2.3);
\draw(0,0.7)--(2,2.3);
\draw (0,0)--(0,3) (2,0)--(2,3);
\end{scope}
\end{tikzpicture}
$$

 We also show that many basic notions from the theory of Poisson-Lie groups / Lie bialgebras have a natural quasi-Poisson version. This is true in particular for Manin triples (that get replaced by quadruples) and for moment maps, for which one can find natural examples using moduli spaces of flat connections on surfaces. The nilpotent radicals of parabolic subgroups are the simplest examples of these quasi-Poisson Lie groups.

\section{Quasi-Poisson manifolds}
Let $\g$ be a Lie algebra with an invariant element $t\in(S^2\g)^\g$. We shall suppose that $t$ is non-degenerate and denote the corresponding bilinear pairing (the inverse of $t$) on $\g$ by $\la\,,\,\ra$. Let $\phi\in\bigwedge^3\g$ be given by
$$\phi(\alpha,\beta,\gamma)=\frac{1}{4}\alpha\bigl([t^\sharp\beta,t^\sharp\gamma]\bigr)\text{ for all }\alpha,\beta,\delta\in\g^*.$$

\begin{defn}[\cite{AKM}]
A \emph{$\g$-quasi-Poisson manifold} is a manifold $M$ with an action $\rho$ of $\g$ and with a $\g$-invariant bivector field $\pi$, satisfying
$$[\pi,\pi]/2=\rho(\phi),$$
where $\rho:\g\to\Gamma(TM)$ is extended to an algebra morphism $\rho:\bigwedge\g\to\Gamma(\bigwedge TM)$. Equivalently, the bracket $\{f,g\}:=\pi(df,dg)$ on $C^\infty(M)$ satisfies
\begin{equation}\label{qJacobi}
\{\{f,g\},h\}+\{\{g,h\},f\}+\{\{h,f\},g\}=\rho(\phi)(df,dg,dh).
\end{equation}
A map $f:M_1\to M_2$ between two $\g$-quasi-Poisson manifolds $(M_1,\rho_1,\pi_1)$ and $(M_2,\rho_2,\pi_2)$ is \emph{quasi-Poisson} if it is $\g$-equivariant and if $f_*\pi_1=\pi_2$.
\end{defn}
If $M$ is $\g$-quasi-Poisson and if the action $\g$ on $M$ integrates to an action of a connected Lie group $G$, we shall say that $M$ is $G$-quasi-Poisson.

%One motivation for $\g$-quasi-Poisson manifolds is the following. There is a deformation of the symmetric monoidal structure on the category of $U\g$-modules (the deformation depends  on a choice of a Drinfeld associator). Suppose that $M$ is a $\g$-manifold and that we deform the product on $C^\infty(M)$ so that it is an associative algebra in the deformed category. Then
%$$\{f,g\}:=(f*g-g*f)/\hbar\mod\hbar,$$
%where $*$ is the deformed product and $\hbar$ the deformation parameter,
%is a $\g$-quasi-Poisson structure on $M$. 
%
%If $\mathscr C$ is a braided monoidal category and $A,B\in\mathscr C$ are associative algebras (monoids) in $\mathscr C$ then $A\otimes B$ is also an associative algebra, with the product given by
%$$
%\begin{tikzpicture}[baseline=1cm]
%\coordinate (diff) at (0.5,0);
%\coordinate (dy) at (0,0.5);
%\node(A1) at (0,0) {$A$};
%\node(B1) at ($(A1)+(diff)$) {$B$};
%\node(A2) at (2,0) {$A$};
%\node(B2) at ($(A2)+(diff)$) {$B$};
%\node(A3) at ($(A1)+(0,2.5)+0.5*(diff)$) {$A$};
%\node(B3) at ($(A2)+(0,2.5)+0.5*(diff)$) {$B$};
%\coordinate(A) at ($(A3)-(dy)$);
%\coordinate(B) at ($(B3)-(dy)$);
%\draw (A2)..controls +(0,1) and +(0,-1)..(A);
%\draw[line width=1ex,white] (B1)..controls +(0,1) and +(0,-1)..(B);
%\draw (B1)..controls +(0,1) and +(0,-1)..(B);
%\draw (A1)..controls +(0,1) and +(0,-1)..(A);
%\draw (B2)..controls +(0,1) and +(0,-1)..(B);
%\draw (B)--(B3);
%\draw(A)--(A3);
%\end{tikzpicture}\ .
%$$
%The corresponding notion in quasi-Poisson geometry is called \emph{fusion}:

There is a natural $\g$-quasi-Poisson structure on the product of two $\g$-quasi-Poisson manifolds:

\begin{defn}[\cite{AKM}]
If $(M_1,\rho_1,\pi_1)$ and $(M_2,\rho_2,\pi_2)$ are $\g$-quasi-Poisson manifolds then their \emph{fusion product} $M_1\circledast M_2$ is the $\g$-quasi-Poisson manifold $M_1\times M_2$ with the diagonal action of $\g$ and with the bivector field
$$\pi=\pi_1+\pi_2-\frac{1}{2}\,t^{ij}\rho_1(e_i)\wedge\rho_2(e_j),$$
where $t=t^{ij}e_i\otimes e_j$ in a basis $e_i$ of $\g$.
\end{defn}

Let us note that the fusion product is associative, but not commutative. It makes the category of $\g$-quasi-Poisson manifolds to a (non-symmetric) monoidal category.

\begin{example}[\cite{LBS}]
A particularly useful class of $\g$-quasi-Poisson manifolds is given by triples $(M,\rho,0)$, where $\rho$ is an action of $\g$ such that $\rho(t)=0\in\Gamma(T^{\otimes 2}M)$, or equivalently such that the stabilizers of the action are coisotropic Lie subalgebras of $\g$. In this case the diagonal map $M\to M\circledast M$ is $\g$-quasi-Poisson.
\end{example}

Another useful construction is the \emph{coisotropic reduction}:
\begin{prop}[\cite{LBS2}]
Let $\mf p\subset\g$ be a coisotropic Lie subalgebra (i.e.\ $\mf p^\perp\subset\mf p$) and let $\g':=\mf p/\mf p^\perp$. If $M':=M/\mf p^\perp$ is a manifold, i.e.\ if there is a manifold $M'$ and a surjective submersion $p:M\to M'$ whose fibers are the $\mf p^\perp$-orbits, then $(M',\rho',\pi')$ is $\g'$-quasi-Poisson, where $\pi'=p_*\pi$ and $\rho'$ is induced from $\rho$. Similarly, if $M''=M/\mf p$ is a manifold then $(M'',\pi'')$ is Poisson, where $\pi''$ is the push-forward of $\pi$.
\end{prop}

The projection $(M_1\circledast_\g M_2)/\mf p^\perp\to(M_1/\mf p^\perp)\circledast_{\g'}(M_2/\mf p^\perp)$ is $\g'$-quasi-Poisson, i.e.\ reduction can be seen as a (colax) monoidal functor.

\section{$\g$-Quasi-Poisson Lie groups and Lie $\g$-quasi-bialgebras}

\begin{defn}
A \emph{$\g$-quasi-Poisson Lie group} is a Lie group $H$ with a $\g$-quasi-Poisson structure $(\rho,\pi)$, such that
the multiplication on $H$ is  a quasi-Poisson map $$H\circledast H\to H.$$
\end{defn}

In more detail it means the following: the $\g$-quasi-Poisson structure $(\rho,\pi)$ is such that the action $\rho$ of $\g$ on $H$ is by infinitesimal automorphisms of the Lie group $H$, and the bivector field $\pi$ satisfies the relation (for any $h,h'\in H$)
\begin{equation}\label{qmult}
\pi|_{hh'}=h\,\pi|_{h'}+\pi|_{h}\,h'-\frac{1}{2}\,t^{ij}(\rho(e_i)|_{h}\,h')\wedge(h\,\rho(e_j)|_{h'}),
\end{equation}
where $\pi|_h\in\bigwedge^2T_hH$ is the value of the bivector field $\pi$ at $h$, and $h\,\pi|_{h'}\in\bigwedge^2T_{hh'}H$ is $\pi|_{h'}$ left-translated by $h$ to $hh'$ (and similarly for $\pi|_{h}\,h'$ and right translation). Setting $h=h'=1$ we get
$$\pi|_1=0.$$

The corresponding infinitesimal notion is as follows.

\begin{defn}
A \emph{Lie $\g$-quasi-bialgebra} is a triple $(\h,\dot\rho,\delta)$, where
\begin{itemize}
\item $\h$ is a Lie algebra
\item $\dot\rho$ is an action of $\g$ on $\h$ by Lie algebra derivations
\item $\delta:\h\to\bigwedge^2\h$ is a $\g$-invariant Lie cobracket on $\h$
\end{itemize}
such that
\begin{equation}\label{mcoc}
\on{ad}^{(2)}_X\delta(Y)-\on{ad}^{(2)}_Y\delta(X)-\delta([X,Y])
-t^{ij}\,\dot\rho(e_i)(X)\wedge\dot\rho(e_j)(Y)=0
\end{equation}
for all $X,Y\in\h$, where $\on{ad}^{(2)}_X:=\on{ad}_X\otimes1+1\otimes\on{ad}_X$.
\end{defn}

\begin{thm}
If $(H,\rho,\pi)$ is a $\g$-quasi-Poisson Lie group then $(\h,\dot\rho,\delta)$ is a Lie $\g$-quasi-bialgebra, with
$$\bigl(\dot\rho(u)X\bigr)^L:=[\rho(u),X^L]\quad\forall u\in\g,X\in\h$$
where, for $X\in\h$, $X^L$ denotes the corresponding left-invariant vector field on $H$, and
$$\delta(X)=[X^L,\pi]|_1,$$
i.e.\ $\delta$ is the linearization of $\pi$ at the unit element $1\in H$.
\end{thm}
\begin{proof}Rewriting \eqref{qmult} as 
$$\pi|_{hh'}\,h'^{-1}-\pi|_h=h\,\pi|_{h'}\,h'^{-1}-\frac{1}{2}\,t^{ij}\rho(e_i)|_{h}\wedge(h\,\rho(e_j)|_{h'}\,h'^{-1}),$$
setting $h'=\exp(\epsilon X)$   and differentiating w.r.t.\ $\epsilon$ at $\epsilon=0$ yields
\begin{equation}\label{Liedpi}
\mathcal L_{X^L}\pi=\delta(X)^L-\frac{1}{2}\,t^{ij}\rho(e_i)\wedge[X^L,\rho(e_j)].
\end{equation}
The identity $\mathcal L_{X^L}\mathcal L_{Y^L}-\mathcal L_{Y^L}\mathcal L_{X^L}=\mathcal L_{[X,Y]^L}$ then gives
$$(\on{ad}^{(2)}_X\delta(Y))^L-(\on{ad}_Y^{(2)}\delta(X))^L-t^{ij}\,[X^L,\rho(e_i)]\wedge [Y^L,\rho(e_j)]=\delta([X,Y])^L,$$
i.e.\ $\delta$ satisfies \eqref{mcoc}.

It remains to prove that $\delta:\h\to\h\otimes\h$ is a Lie cobracket, i.e.\ that its transpose $\delta^t:\h^*\otimes\h^*\to\h^*$ is a Lie bracket. To see it, notice that $\delta^t$ is the linearization of the quasi-Poisson bracket $\{f,g\}:=\pi(df,dg)$ at $1\in H$. The Jacobi identity for $\delta^t$ then follows from the quasi-Jacobi identity \eqref{qJacobi} and from the fact that the linearization of the trivector field $\rho(\phi)$  at $1\in H$ vanishes.
\end{proof}

%\begin{rem}
%If we choose a Drinfeld associator $\Phi$ then the category of $U\g$-modules turns to a braided monoidal category. Suppose that $H$ is a Hopf algebra in this category, in particular the product and the coproduct in $H$ satisfy
%\begin{equation}\label{eq:hopf}
%\begin{tikzpicture}[scale=0.5]
%\draw(0,0)--(1,1)--(2,0) (1,1)--(1,2) (0,3)--(1,2)--(2,3);
%\node at (3,1.5) {=};
%\begin{scope}[xshift=4cm]
%\draw (2,1)--(0,2);
%\draw[white,line width=5] (0,1)--(2,2);
%\draw(0,1)--(2,2);
%\draw (0,0)--(0,3) (2,0)--(2,3);
%\end{scope}
%\end{tikzpicture}
%\end{equation}
%If $H$ is a deformation of $U\h$, where $\h$ is a Lie algebra with an action of $\g$, then the coproduct on $H$ gives a Lie cobracket on $\h$ satisfying \eqref{eq:glb}, i.e.\ $\h$ is a Lie $\g$-bialgebra. 
%\end{rem}
%
%
%
%
%
%A \emph{Lie $\g$-bialgebroid} is a vector bundle $A\to M$ with an action of $\g$, with a Gerstenhaber bracket $[\,,\,]$ and with a differential $d$ on the algebra $\Gamma(\bigwedge A)$ (both $\g$-invariant), such that
%$$d[\alpha,\beta]-[d\alpha,\beta]+(-1)^{|\alpha|}[\alpha,d\beta]=t^{ij} (e_i\cdot \alpha)\wedge(e_j\cdot \beta).$$
%Lie $\g$-bialgebroids correspond to $G$-quasi-Poisson Lie groupoids.
%
%\begin{example}
%If $A=TM$ with its Schouten bracket $[\,,\,]$ and $\g$ acts on $M$ then a differential $d$ making $\Gamma(\bigwedge TM)$ to a Lie $\g$-bialgebroid is equivalent to a $\g$-quasi-Poisson structure on $M$ via
%$$d=[\pi,\cdot]+\frac{1}{2}t^{ij}e_i\wedge[e_j,\cdot].$$
%\end{example}

\section{Manin quadruples}

Lie bialgebras can be conveniently rephrased in terms of Manin triples. The corresponding notion for Lie $\g$-quasi-bialgebras is as follows.

\begin{defn}
$(\mathfrak{d},\mathfrak{a},\mathfrak{b},\mathfrak{c})$ is a \emph{Manin quadruple} if $\mathfrak{d}$ is a Lie algebra with an invariant nondegenerate symmetric pairing $\la,\ra$ and $\mathfrak{a}$, $\mathfrak{b}$, $\mathfrak{c}$ are its Lie subalgebras such that
\begin{itemize}
  \item $\mathfrak{d}=\mathfrak{a}\oplus\mathfrak{b}\oplus\mathfrak{c}$ (as vector spaces)
  \item $\mf a^\perp=\mathfrak{a}\oplus\mathfrak{b}, \quad \mf c^\perp=\mathfrak{b}\oplus\mathfrak{c}$.
\end{itemize}
\end{defn}

\begin{example}\label{ex:nilp}
Let $\mf d$ be a semisimple Lie algebra and $\mf p_+,\mf p_-\subset\mf d$ a pair of opposite parabolic Lie subalgebras. Let $\mf n_\pm=\mf p_\pm^\perp\subset\mf p_\pm$ be the nilpotent radicals of $\mf p_\pm$. Then
$$(\mf d,\mf n_+,\mf p_+\cap\mf p_-,\mf n_-)$$
is a Manin quadruple.
\end{example}

\begin{prop}
If $(\mf d,\mf a,\mf b,\mf c)$ is a Manin quadruple then $[\mf b,\mf a]\subset\mf a$ and $[\mf b,\mf c]\subset\mf c$. As a result, $\mf a\oplus \mf b\subset\mf d$ is a Lie subalgebra which is a semidirect sum of $\mf a$ and $\mf b$. The restriction of $\la,\ra$ to $\mf b$ is non-degenerate, and $\mf a,\mf c\subset\mf d$ are isotropic subspaces. The pairing $\la,\ra$ provides an isomorphism $\mf a^*\cong\mf c$. 
\end{prop} 
\begin{proof}
We have $\la[\mf b,\mf a],\mf a\ra=\la\mf b,[\mf a,\mf a]\ra\subset\la\mf b,\mf a\ra=0$, and similarly $\la[\mf b,\mf a],\mf b\ra=\la\mf a,[\mf b,\mf b]\ra\subset\la\mf a,\mf b\ra=0$, which implies $[\mf b,\mf a]\subset(\mathfrak{a}\oplus\mathfrak{b})^\perp=\mathfrak{a}$, as we wanted to show. The rest of the proposition follows easily.
\end{proof}

\begin{thm}
There is an equivalence between Lie $\g$-quasi-bialgebras and Manin quadruples, given as follows: if $(\mf d,\mf h,\mf g,\mf h^*)$ is a Manin quadruple then $(\h,\dot\rho,\delta)$ is a Lie $\g$-quasi-bialgebra, where
$$\dot\rho(v)X=[v,X]\quad(v\in\g,X\in\h)$$
$$\la\delta(X),A\otimes B\ra=\la X,[B,A]\ra\quad(X\in\h,A,B\in\h^*).$$
\end{thm}
\begin{proof}
If $(\mf d,\mf h,\mf g,\mf h^*)$ is a Manin quadruple then $\delta$ is clearly a Lie cobracket and the action $\dot\rho$ of $\g$ on $\h$ preserves both the bracket and the cobracket on $\h$. To show that $(\h,\dot\rho,\delta)$ is a Lie $\g$-quasi-bialgebra we thus need to verify the relation \eqref{mcoc}. In fact, this relation is simply the Jacobi identity
\begin{equation}\label{jac-coc}
\la[A,X],[Y,B]\ra-\la[B,X],[Y,A]\ra-\la[X,Y],[B,A]\ra=0
\end{equation}
for $X,Y\in\h$, $A,B\in\h^*$.
%To do it, let us notice that the tensor of structure constants $c_{\mf d}\in\bigwedge^3\mf d^*$,
%$$c_{\mf d}(X,Y,Z):=\la[X,Y],Z\ra\quad(\forall X,Y,Z\in\mf d)$$
%decomposes as
%\begin{equation}\label{cd}
%c_{\mf d}=\tilde\delta+\delta+\dot\rho+c_{\mf g},
%\end{equation}
%where $\tilde\delta\in\bigwedge^2\h^*\otimes\h$ is the Lie bracket on $\mf h$, $\delta\in\bigwedge^2\h\otimes\h^*$ is the Lie bracket on $\mf h^*$,  $\dot\rho\in\g^*\otimes\h^*\otimes\h$ is the adjoint action of $\g$ on $\h$,  and finally $c_{\mf g}\in\bigwedge^3\g^*$ is again given by $c_\g(u,v,w)=\la[u,v],w\ra$ for any $u,v,w\in\g$.

To see it, it is convenient to use the graphical notation
$$\la [x,y],z\ra=
\begin{tikzpicture}[baseline=-0.1cm,scale=0.5]
\draw[thick] (-1,0)--(0,0)--(0.5,0.866) (0,0)--(0.5,-0.866);
\fill[white](-1,0)circle[radius=1em] (0.5,0.866)circle[radius=1em] (0.5,-0.866)circle[radius=1em];
\draw(-1,0)node{$x$} (0.5,0.866)node{$z$}
(0.5,-0.866)node{$y$};
\end{tikzpicture}\qquad
\la[x,y],[z,w]\ra=
\begin{tikzpicture}[baseline=-0.1cm,scale=0.5]
\draw[thick] (-1.5,-0.866)--(-1,0) (-1.5,0.866)--(-1,0)--(0,0)--(0.5,0.866) (0,0)--(0.5,-0.866);
\fill[white](-1.5,-0.866)circle[radius=1em] (-1.5,0.866)circle[radius=1em] (0.5,0.866)circle[radius=1em] (0.5,-0.866)circle[radius=1em];
\draw(-1.5,0.866)node{$x$} (-1.5,-0.866)node{$y$}
(0.5,-0.866)node{$z$} (0.5,0.866)node{$w$}; 
\end{tikzpicture}
\quad(x,y,z,w\in\mf d).
$$
In particular, if $X,Y\in\h$ and $A,B\in\h^*$, we get
$$
\la[A,X],[Y,B]\ra=
\begin{tikzpicture}[baseline=0.4cm]
\draw[thick] (0,0)node[below]{$X$}--(0.2,0.5)--(0,1)node[above]{$A$} (1,0)node[below]{$Y$}--(0.8,0.5)--(1,1)node[above]{$B$} (0.2,0.5)--(0.8,0.5);
\draw (1.5,0.5);
\end{tikzpicture}
=
\begin{tikzpicture}[baseline=0.4cm]
\draw[->-=0.6](0,0)node[below]{$X$}--(0.2,0.5);
\draw[->-=0.6] (0.2,0.5)--(0,1)node[above]{$A$};
\draw[->-=0.6](0.2,0.5)--(0.8,0.5);
\draw[->-=0.6](1,0)node[below]{$Y$}--(0.8,0.5);
\draw[->-=0.6](0.8,0.5)--(1,1)node[above]{$B$};
\end{tikzpicture}
+
\begin{tikzpicture}[baseline=0.4cm]
\draw[->-=0.6](0,0)node[below]{$X$}--(0.2,0.5);
\draw[->-=0.6] (0.2,0.5)--(0,1)node[above]{$A$};
\draw[->-=0.6](0.8,0.5)--(0.2,0.5);
\draw[->-=0.6](1,0)node[below]{$Y$}--(0.8,0.5);
\draw[->-=0.6](0.8,0.5)--(1,1)node[above]{$B$};
\end{tikzpicture}
+
\begin{tikzpicture}[baseline=0.4cm]
\draw[->-=0.6](0,0)node[below]{$X$}--(0.2,0.5);
\draw[->-=0.6] (0.2,0.5)--(0,1)node[above]{$A$};
\draw[decorate,decoration={snake,amplitude=.4mm,segment length=1mm}](0.2,0.5)--(0.8,0.5);
\draw[->-=0.6](1,0)node[below]{$Y$}--(0.8,0.5);
\draw[->-=0.6](0.8,0.5)--(1,1)node[above]{$B$};
\end{tikzpicture}
$$
where directed edges go from $\h$ to $\h^*$ and the wiggly edge connects $\g$ with $\g$. The identity \eqref{jac-coc}
thus gives us
\begin{multline*}
\biggl(
\begin{tikzpicture}[baseline=0.4cm]
\draw[->-=0.6](0,0)node[below]{$X$}--(0.2,0.5);
\draw[->-=0.6] (0.2,0.5)--(0,1)node[above]{$A$};
\draw[->-=0.6](0.8,0.5)--(0.2,0.5);
\draw[->-=0.6](1,0)node[below]{$Y$}--(0.8,0.5);
\draw[->-=0.6](0.8,0.5)--(1,1)node[above]{$B$};
\end{tikzpicture}
-
\begin{tikzpicture}[baseline=0.4cm]
\draw[->-=0.6](0,0)node[below]{$X$}--(0.2,0.5);
\draw[->-=0.6] (0.2,0.5)--(0,1)node[above]{$B$};
\draw[->-=0.6](0.8,0.5)--(0.2,0.5);
\draw[->-=0.6](1,0)node[below]{$Y$}--(0.8,0.5);
\draw[->-=0.6](0.8,0.5)--(1,1)node[above]{$A$};
\end{tikzpicture}
\biggr)
+
\biggl(
\begin{tikzpicture}[baseline=0.4cm]
\draw[->-=0.6](0,0)node[below]{$X$}--(0.2,0.5);
\draw[->-=0.6] (0.2,0.5)--(0,1)node[above]{$A$};
\draw[->-=0.6](0.2,0.5)--(0.8,0.5);
\draw[->-=0.6](1,0)node[below]{$Y$}--(0.8,0.5);
\draw[->-=0.6](0.8,0.5)--(1,1)node[above]{$B$};
\end{tikzpicture}
-
\begin{tikzpicture}[baseline=0.4cm]
\draw[->-=0.6](0,0)node[below]{$X$}--(0.2,0.5);
\draw[->-=0.6] (0.2,0.5)--(0,1)node[above]{$B$};
\draw[->-=0.6](0.2,0.5)--(0.8,0.5);
\draw[->-=0.6](1,0)node[below]{$Y$}--(0.8,0.5);
\draw[->-=0.6](0.8,0.5)--(1,1)node[above]{$A$};
\end{tikzpicture}
\biggr)\\
-
\begin{tikzpicture}[baseline=0.4cm]
\draw[->-=0.6](0,0)node[below]{$X$}--(0.5,0.3);
\draw[->-=0.6](1,0)node[below]{$Y$}--(0.5,0.3);
\draw[->-=0.6](0.5,0.3)--(0.5,0.7);
\draw[->-=0.6] (0.5,0.7)--(0,1)node[above]{$A$};
\draw[->-=0.6] (0.5,0.7)--(1,1)node[above]{$B$};
\end{tikzpicture}
+
\biggl(
\begin{tikzpicture}[baseline=0.4cm]
\draw[->-=0.6](0,0)node[below]{$X$}--(0.2,0.5);
\draw[->-=0.6] (0.2,0.5)--(0,1)node[above]{$A$};
\draw[decorate,decoration={snake,amplitude=.4mm,segment length=1mm}](0.2,0.5)--(0.8,0.5);
\draw[->-=0.6](1,0)node[below]{$Y$}--(0.8,0.5);
\draw[->-=0.6](0.8,0.5)--(1,1)node[above]{$B$};
\end{tikzpicture}
-
\begin{tikzpicture}[baseline=0.4cm]
\draw[->-=0.6](0,0)node[below]{$X$}--(0.2,0.5);
\draw[->-=0.6] (0.2,0.5)--(0,1)node[above]{$B$};
\draw[decorate,decoration={snake,amplitude=.4mm,segment length=1mm}](0.2,0.5)--(0.8,0.5);
\draw[->-=0.6](1,0)node[below]{$Y$}--(0.8,0.5);
\draw[->-=0.6](0.8,0.5)--(1,1)node[above]{$A$};
\end{tikzpicture}
\biggr)
=0
\end{multline*}
which is the relation \eqref{mcoc} contracted with $A\otimes B$ (the terms are parenthesized to correspond to the 4 terms of \eqref{mcoc}).

This construction is easily seen to be reversible: if $(\h,\dot\rho,\delta)$ is a Lie $\g$-quasi-bialgebra then we get a Lie bracket on $\mf d$, as \eqref{jac-coc} is the non-trivial part of the Jacobi identity in $\mf d$.
\end{proof}

\section{Construction of $\g$-quasi-Poisson groups}

Let $(\mf d,\h,\g,\h^*)$ be a Manin quadruple and let $H$ be the 1-connected group integrating $\h$ (or at least such that the adjoint action of $H$ on $\mf d$ is well-defined). In this section we shall describe how to make $H$ to a $\g$-quasi-Poisson group.

The adjoint action of $\g$ on $\h$ extends to an action $\rho$ of $\g$ on $H$ by infinitesimal automorphisms. More generally,
for any $v\in\mf d$ let us define a vector field $\hat\rho(v)$ on $H$ via
$$\hat\rho(v)|_h:=p(\on{Ad}_h v)\,h\qquad(h\in H)$$
where $p:\mf d\to\h$ is the projection w.r.t.\ $\g\oplus\h^*$. By construction
$$\hat\rho(X)=X^L\ \ \text{for }X\in\h,\quad\hat\rho(v)=\rho(v)\ \ \text{for }v\in\g.$$

\begin{prop}\label{prop:hat-rho}
$\hat\rho$ is an action of $\mf d$ on $H$.
\end{prop}
\begin{proof}
Let $D$ be the 1-connected Lie group integrating $\mf d$ and let $P_-\subset D$ be the connected Lie subgroup integrating the Lie subalgebra $\mf p_-:=\g\oplus\h^*\subset\mf d$. Let us first suppose that $P_-$ is closed, so that the quotient $P_-\backslash D$ is a manifold. The action of $\mf d$ on $D$ by the left-invariant vector fields projects to an action of $\mf d$ on $P_-\backslash D$. Moreover the composition $H\to D\to P_-\backslash D$ is a local diffeomorphism, so the action of $\mf d$ on $P_-\backslash D$ pulls back to an action on $H$. By construction, this action is $\hat\rho$.

This argument is easily refined to the case when $P_-\subset D$ is not closed. The cosets $P_-\,d$, $d\in D$, form a foliation $\mathcal F$ of $D$. For any $h\in H$ let us choose an open subset $h\in U_h\subset D$ with the property that $H\cap U_h$ intersects every leaf of $\mathcal F$ once, so that we have a diffeomorphism $H\cap U_h\cong U_h/\mathcal F$. The Lie algebra $\mf d$ acts on $U_h/\mathcal F$ and under the diffeomorphism $H\cap U_h\cong U_h/\mathcal F$ this action is equal to $\hat\rho$.
\end{proof}

\begin{thm}\label{thm:construction-of-H}
Let $(\mf d,\h,\g,\h^*)$ be a Manin quadruple, $(\h,\dot\rho,\delta)$ the corresponding Lie $\g$-quasi-bialgebra, and $H$ the 1-connected Lie group integrating $\h$. There is a unique $\g$-quasi-Poisson structure $(\rho,\pi)$ on $H$ making $H$ to a $\g$-quasi-Poisson group, such that the corresponding Lie $\g$-quasi-bialgebra is  $(\h,\dot\rho,\delta)$. Explicitly,
\begin{equation}\label{pi-expl}
\pi=\frac{1}{2}\,\hat\rho(E^i)\wedge (E_i)^L,
\end{equation}
where $E_i$ is a basis of $\h$ and $E^i$ the dual basis of $\h^*$. 
%Equivalently,
%$$H\cong(\hat H \circledast_{\mf d} \hat H )/H$$
%where $\hat H$ is the $\mf d$-quasi-Poisson manifold $(H,\hat\rho,0)$ and the isomorphism is given by $(h_1,h_2)\mapsto h_1h_2^{-1}$.
\end{thm}
\begin{proof}
Uniqueness follows from $\pi|_1=0$ and from \eqref{Liedpi}.

Let us note that $\hat H:=(H,\hat\rho,0)$ is  a $\mf d$-quasi-Poisson manifold, as the stabilizer of $\hat\rho$ at any point $h\in H$ is coisotropic, namely $\on{Ad}_{h^{-1}}\mf p_-$, where $\mf p_-=\g\oplus\h^*\subset\mf d$.  The diagonal map $\hat H\to\hat H\circledast_{\mf d}\hat H$ is $\mf d$-quasi-Poisson.

Let $\mf p_+:=\h\oplus\g\subset\mf d$; we have $\h=(\mf p_+)^\perp\subset\mf p_+$ and $\mf p_+/\h=\g$. Using the coisotropic reduction by $\h=(\mf p_+)^\perp\subset\mf p_+$ we turn
$$(\hat H \circledast_{\mf d} \hat H )/H$$
to a $\g$-quasi-Poisson manifold. 

Let us identify $(\hat H \circledast_{\mf d} \hat H )/H$ with $H$ via $(h_1,h_2)\mapsto h_1 h_2^{-1}$. We thus constructed a $\g$-quasi-Poisson structure $(\rho,\pi)$ on $H$. Let us check that it makes $H$ to a $\g$-quasi-Poisson group, i.e.\ that the group product $H\times H\to H$ is a $\g$-quasi-Poisson map. To see it, consider the $\g$-quasi-Poisson maps
$$
\begin{tikzcd}
(\hat H\circledast_{\mf d}\hat H\circledast_{\mf d}\hat H)/H\arrow{r}{\Delta}\arrow{d}{\epsilon} & (\hat H\circledast_{\mf d}\hat H\circledast_{\mf d}\hat H\circledast_{\mf d}\hat H)/H\arrow{d}{\kappa} \\
(\hat H\circledast_{\mf d}\hat H)/H & \bigl((\hat H \circledast_{\mf d} \hat H )/H\bigr)\circledast_\g\bigl((\hat H \circledast_{\mf d} \hat H )/H\bigr)\arrow[swap,dashed, near start]{l}{\epsilon \circ (\kappa \circ \Delta)^{-1}\ }
\end{tikzcd}
$$
given by
$$\epsilon([h_1,h_2,h_3])=[h_1,h_3]$$
$$\Delta([h_1,h_2,h_3])=[h_1,h_2,h_2,h_3]$$
$$\kappa([h_1,h_2,h_3,h_4])=([h_1,h_2],[h_3,h_4])$$
  where $h_1,h_2,h_3,h_4\in H$ and the square bracket denotes the respective  cosets. It can be easily seen that $\kappa \circ \Delta$ is in fact a diffeomorphism. Moreover, since all the maps at the diagram are $\mathfrak{g}$-quasi-Poisson maps, we have that $\epsilon \circ (\kappa \circ \Delta)^{-1}$, which coincides with the multiplication $H\times H\to H$ under our identification $H\cong(\hat H \circledast_{\mf d} \hat H )/H$, is also $\mathfrak{g}$-quasi-Poisson. 
  
Let us now prove that the bivector field $\pi$ on $H=(\hat H\circledast_{\mf d}\hat H)/H$ satisfies \eqref{pi-expl}. Let $e_\alpha$ be a basis of $\g$ and $e^\alpha$ the dual basis of $\g$ (w.r.t.\ $\la,\ra$). We have
$$t_{\mf d}=E^i\otimes E_i+E_i\otimes E^i + e_\alpha\otimes e^\alpha$$
and thus
$$\pi=-\frac{1}{2}\tau_*\bigl(\hat\rho_1(E^i)\wedge\hat\rho_2(E_i)+\hat\rho_1(E_i)\wedge\hat\rho_2(E^i)+\hat\rho_1(e_\alpha)\wedge\hat\rho_2(e^\alpha)\bigr)$$
where $\tau:H\times H\to H$, $\tau(h_1,h_2)=h_1h_2^{-1}$ is our identification of $H=(\hat H\circledast_{\mf d}\hat H)/H$ with $H$. Since $\tau(h,1)=h$ and $\hat\rho(\h^*)|_1=\hat\rho(\g)|_1=0$ and $\hat\rho(X)=X^L$ for $X\in\h$, we get
$$\pi|_h=-\frac{1}{2}\tau_*\bigl(\hat\rho_1(E^i)|_h\wedge \hat\rho_2(E_i)|_1\bigr)=\frac{1}{2}\,\hat\rho(E^i)|_h\wedge h\,E_i$$
as we wanted to show.

To finish the proof, we need to show that $[X^L,\pi]|_1=\delta(X)$, i.e.
$$\la[X^L,\pi]|_1,A\otimes B\ra=\la X,[B,A]\ra\qquad\forall X\in\h,\;A,B\in\h^*.$$
Using \eqref{pi-expl} we get
\begin{multline}
\la[X^L,\pi]|_1,A\otimes B\ra=\frac{1}{2}\bigl(\la[X,E^i],A\ra\la E_i,B\ra-\la[X,E^i],B\ra\la E_i,A\ra\bigr)\\
=\frac{1}{2}\bigl(\la[X,B],A\ra-\la[X,A],B\ra\bigl)=\la X,[B,A]\ra
\end{multline}
as we needed.
\end{proof}

\begin{rem}
The proof above is not the shortest possible, but it contains a conceptual  construction of the $\g$-quasi-Poisson group $H$ as $(\hat H\circledast_{\mf d}\hat H)/H$ which will be useful when we consider moment maps and deformation quantization. This construction is of interest also in the special case of $\g=0$, i.e.\ of Poisson-Lie groups.
\end{rem}

\begin{example}
Let us consider the Manin quadruple of Example \ref{ex:nilp} in the case of $\mf d=\mf{sl}(N)$, with the inner product $\la A,B\ra=\on{Tr} AB$. In this case $\g$ consists of the traceless block-diagonal $N\times N$-matrices (for some partition of $(1,2,\dots,N)$ into blocks), $\h$ of matrices with blocks above the diagonal, and $\h^*$ of matrices with blocks below the diagonal. We have $H=1+\h\subset SL(N)$. For an element $X\in H$ let $x^k_l$ denote the matrix elements of $X$; $x^k_l$'s are functions on $H$. Equation \eqref{pi-expl} gives the following expression for the $\g$-quasi-Poisson bracket on $H$:
$$
\{x^k_l,x^m_n\}=\frac{1}{2}\Bigl(x^k_n x^m_l\bigl(\theta^{\dot{l}}_{\dot{n}}-\theta^{\dot{n}}_{\dot{l}}+\theta^{\dot{k}}_{\dot{m}}-\theta^{\dot{m}}_{\dot{k}}\bigr)+\delta^k_n \sum_r \delta^{\dot{r}}_{\dot{n}} x^m_rx^r_l-\delta^m_l \sum_r \delta^{\dot{r}}_{\dot{l}} x^k_rx^r_n\Bigr).
$$
where
$$
\theta^m_n=
\begin{cases}
1&m<n\\
0&m\geq n
\end{cases}
$$
and $\dot n$, for $1\leq n\leq N$, denotes the number of the block to which $n$ belongs.
\end{example}

\section{Moment maps}

In this section we shall see that the theory of moment maps, as known from the case of Lie bialgebras and Poisson-Lie groups, can be developed also for Lie $\g$-quasi-bialgebras and \gqP\ Lie groups.

\begin{defn}
If $(H,\rho_H,\pi_H)$ is a \gqP\ group and $(M,\rho_M,\pi_M)$ a \gqP\ manifold, an action $H\times M\to M$ is \emph{\gqP} if it is a \gqP\ map 
$$H\circledast M\to M.$$

 If $(\h,\dot\rho,\delta)$ is a Lie $\g$-quasi-bialgebra, an action $\psi$ of $\h$ on $M$ is \emph{\gqP} if it is $\g$-equivariant and if
\begin{equation}\label{inf-action}
[\psi(X),\pi_M]=-\psi(\delta(X))+\half t^{ij}\,\psi(\dot\rho(e_i)X)\wedge\rho_M(e_j)
\end{equation} for every $X\in\h$.
\end{defn}

\begin{prop}
If $(H,\rho_H,\pi_H)$ is a \gqP\ group and $(M,\rho_M,\pi_M)$ a \gqP\ manifold, and  $\cdot:H\times M\to M$ is a $\g$-equivariant action of $H$ on $M$, then this action is \gqP\  iff the corresponding action $\psi$ of $\h$, $\psi(X)|_m:=\left.\frac{d}{dt}\right|_0 e^{-tX}\cdot m$ ($m\in M$), is \gqP.
\end{prop}
\begin{proof}
Let us define an action $*$ of $H$ on $\bigwedge^2 TM$: for $h\in H$, $m\in M$, $b_m\in\bigwedge^2T_mM$
$$ h*b_m:=h\cdot b_m + \pi_H|_h\cdot m-\half t^{ij} (\rho_H(e_i)|_h\cdot m)\wedge (h\cdot \rho_M(e_j)|_m)\in{\textstyle\bigwedge}^2T_{h\cdot m}M.$$
The fact that $*$ is an action follows easily from the multiplicativity property \eqref{qmult} of $\pi_H$ and from the $\g$-equivariance of the action $\cdot:H\times M\to M$ and of the product $H\times H\to H$.

The action $*$ has the property that the map $\cdot:H\circledast M\to M$ is \gqP\ iff the bivector field $\pi_M$ is $H$-invariant under the $*$-action, i.e.\ iff $\pi_M$ is invariant under the corresponding action of the Lie algebra $\h$. The latter invariance is Equation \eqref{inf-action}.
\end{proof}

As observed in \cite{LBS}, if $(M,\rho,\pi)$ is a $\g$-quasi-Poisson manifold then the operators
$$d_{\pm}:=\pm[\pi,\cdot]+\half t^{ij}\,\rho(e_i)\wedge[\rho(e_j),\cdot]$$
are differentials on the graded algebra $\Gamma(\bigwedge TM)$ of polyvector fields on $M$.
Comparing with \eqref{inf-action} we see that a $\g$-equivariant action $\psi$ of $\h$ on $M$ is \gqP\ iff 
$$d_-\psi(X)=-\psi(\delta(X)).$$

Let us recall that if $A\to M$ is a vector bundle and if $d_A$ is a differential on the graded algebra $\Gamma(\bigwedge A^*)$ then $A$ is a Lie algebroid with the bracket
$$\la[u,v]_A,\omega\ra=[d_A i_u+i_u d_A,i_v]\omega\qquad u,v\in\Gamma(A),\omega\in\Gamma(A^*)$$
and with the anchor $\mathsf{a}:A\to TM$
$$\mathsf{a}(u) f=\la u,d_A f\ra\qquad u\in\Gamma(A), f\in C^\infty (M).$$
In particular, $d_\pm$ give us two Lie algebroid structures $([,]_+,\mathsf a_+)$ and $([,]_-,\mathsf a_-)$ on $T^*M$. The anchors are
$$\mathsf{a}_+(\alpha)=\sigma(\alpha,\cdot),\quad\mathsf{a}_-(\alpha)=\sigma(\cdot,\alpha)$$
where
$$\sigma:=\pi+\half \rho(t)\in\Gamma(T^{\otimes 2}M).$$

\begin{prop}\label{prop:Lie-H-star}
Let $(H,\rho,\pi)$ be a $\g$-quasi-Poisson Lie group. Then for any $\alpha,\beta\in\h^*$
$$[\alpha^L,\beta^L]_+=-[\alpha,\beta]^L.$$
\end{prop}
\begin{proof}
For any $X\in\h$ we have, by \eqref{Liedpi},
$$d_+X^L=-\delta(X)^L$$
and thus
\begin{multline*}
\la[\alpha^L,\beta^L]_+,X^L\ra=[d_+i_{\alpha^L}+i_{\alpha^L} d_+,i_{\beta^L}]X^L\\
=-i_{\beta^L}i_{\alpha^L}d_+X^L=\la\alpha\otimes\beta,\delta(X)\ra=-\la[\alpha,\beta],X\ra.
\end{multline*}
Since this equality holds for every $X\in\h$, we have indeed $[\alpha^L,\beta^L]_+=-[\alpha,\beta]^L$.
\end{proof}

If $(\mf d,\h,\g,\h^*)$ is a Manin quadruple then both $\h$ and $\h^*$ are $\g$-quasi-bialgebras; we shall call $\h^*$ \emph{the dual Lie $\g$-quasi-bialgebra of $\h$}.
We can now formulate and prove the main result of this section.

\begin{thm}[Moment map]
Let $(\h,\dot\rho,\delta)$ be a \Lgqb\ and let $H^*$ be the 1-connected \gqP\ Lie group integrating its dual. If
$$\mu:M\to H^*$$
is a \gqP\ map then the map
$$\psi:\h\to\Gamma(TM),\quad\psi(X)=-\mathsf a_+(\mu^*X^L)$$
is a \gqP\ action of $\h$ on $M$.
\end{thm}
\begin{proof}
Let us extend the Lie brackets $[,]_\pm$ to Gerstenhaber brackets $[,]_\pm$ on the graded algebras $\Omega(M)$ and $\Omega(H^*)$. The morphisms of graded algebras
\begin{equation}\label{dgerst}
(\Omega(H^*),d,[,]_+)\xrightarrow{\mu^*}(\Omega(M),d,[,]_+)\xrightarrow{\bigwedge\mathsf a_+}(\Gamma(\textstyle\bigwedge TM),d_-,[,])
\end{equation}
preserve the indicated differentials and Gerstenhaber brackets: for the first map it follows immediately from the fact that $\mu^*$ is \gqP. For the second map it follows from the more general fact that if $(A,[,]_A,\mathsf a)$ is a Lie algebroid over $M$ then $\bigwedge\mathsf a:(\Gamma(\bigwedge A),[,]_A)\to(\Gamma(\bigwedge TM),[,])$ is  a morphism of Gerstenhaber algebras, and $\bigwedge\mathsf a^t:(\Omega(M),d)\to(\Gamma(\bigwedge A^*),d_A)$ is a morphism of differential graded algebras, where $\mathsf a^t:T^*M\to A^*$ is the transpose of $\mathsf a$, together with the fact that $\mathsf a_-^t=\mathsf a_+$.

Let now $\Psi:(\Omega(H^*),d,[,]_+)\to(\Gamma(\textstyle\bigwedge TM),d_-,[,])$ be the composition of \eqref{dgerst}. By definition we have, for $X\in\h$, that $\psi(X)=-\Psi(X^L)$. The fact that $\psi$ is an action of $\h$ now follows from Proposition \ref{prop:Lie-H-star}, and the fact that this action is \gqP, i.e.\ that $d_-\psi(X)=-\psi(\delta(X))$, follows from $d(X^L)=(\delta(X))^L$.
\end{proof}

\section{Examples from moduli spaces of flat connections}

  One  important class of $\mathfrak{g}$-quasi-Poisson manifolds arises from moduli spaces of flat connections on surfaces with marked points on the boundary. It is also this place where one can find many examples of momentum maps, which will be our focus in this section.

  Firstly, let us mention a more general notion of fusion, which we will need in this section.
  \begin{defn}[\cite{AKM}]
  Let $(\g,t)$, $(\g',t')$ be two Lie algebras with invariant inner products. If $(M,\rho,\pi)$ is a $\g\oplus \g\oplus\g'$-quasi-Poisson manifold, then the (internal) fusion of $M$ is the $\g\oplus\g'$-quasi-Poisson manifold $(M,\rho_\circledast,\pi_\circledast)$ given by
  $$\rho_\circledast(X,Y)=\rho(X,X,Y)\quad \forall X\in \g,\, Y\in \g',$$
  $$\pi_\circledast=\pi-\frac{1}{2}(\rho_1\wedge \rho_2)(t),$$
  where $\rho_1$ and $\rho_2$ are the actions of the first and of the second copy of $\g$ in $\g\oplus \g\oplus\g'$, respectively.
  \end{defn}

Let now $\Sigma$ be an compact oriented surface with boundary (not necessarily connected) and let $V\subset\partial\Sigma$ be a finite set meeting every component of $\Sigma$. Let, as above, $\mf d$ be a Lie algebra with invariant inner product, and $D$ a Lie group integrating $\mf d$. The moduli space
$$M_D(\Sigma,V)$$
 of flat principal $D$-bundles over $\Sigma$ trivialized over $V$ is naturally identified with the space of groupoid morphisms
$$\Hom(\Pi_1(\Sigma,V),D).$$
There is a natural action of the group $D^V$ on $M_D(\Sigma,V)$ given by changing the trivialization of a principal $D$-bundle $P\to\Sigma$ over $V$.

\begin{thm}[\cite{LBS2}]\label{thm:moduli}
There is a natural $D^V$-quasi-Poisson structure $\pi_{\Sigma,V}$ on $M_D(\Sigma,V)$ characterized by these properties:
\begin{enumerate}
\item If $\Sigma$ is a disk and $V$ consists of two points then $\pi_{\Sigma,V}=0$.
\item If $\Sigma=\Sigma_1\sqcup\Sigma_2$ (and so $M_D(\Sigma,V)=M_D(\Sigma_1,V_1)\times M_D(\Sigma_1,V_1)$) then 
$\pi_{\Sigma,V}=\pi_{\Sigma_1,V_1}+\pi_{\Sigma_2,V_2}$.
\item Let $x,y\in V$ and let $\Sigma'$ be obtained from $\Sigma$ by a ``corner connected sum'':
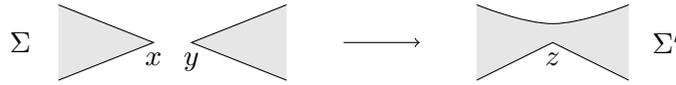
$$
\begin{tikzpicture}[scale=0.5]
\fill[color=white!90!black] (0,0.5) -- (2.5,1.5) -- (0,2.5);
\fill[color=white!90!black] (6,0.5) -- (3.5,1.5) -- (6,2.5);
\draw (0,0.5) -- (2.5,1.5) -- (0,2.5);
\draw (6,0.5) -- (3.5,1.5) -- (6,2.5);
\node[below] at (2.5,1.5) {$x$};
\node[below] at (3.5,1.5) {$y$};
\draw[->] (7.5,1.5) -- (9.5,1.5);
\fill[color=white!90!black]  (11,0.5) -- (13,1.5) -- (15,0.5)-- plot[smooth, tension=.7] coordinates {(15,2.5) (13,2) (11,2.5)};
\draw (11,0.5) -- (13,1.5) -- (15,0.5);
\draw  plot[smooth, tension=.7] coordinates {(15,2.5) (13,2) (11,2.5)};
\node[below] at (13,1.5) {$z$};
\node at (-1,1.5) {$\Sigma$};
\node at (16,1.5) {$\Sigma'$};
\end{tikzpicture}
$$
Let $V'$ be the image of $V$ in $\Sigma'$, i.e.\ with $x$ and $y$ identified (denoted $z$ on the picture). In this case $M_D(\Sigma', V')\cong M_D(\Sigma, V)$ (the isomorphism is induced by the gluing map $\Sigma\to\Sigma'$). Then the $D^{V'}$-quasi-Poisson structure on $M_D(\Sigma', V')$ is obtained from  the $D^V$-quasi-Poisson structure on $M_D(\Sigma, V)$ by fusion of the $\mf d$'s acting at $x$ and $y$.
\end{enumerate}
This quasi-Poisson structure is functorial under embeddings of surfaces: if $i:\Sigma'\to\Sigma$ is an embedding then the induced map $i^*:M_D(\Sigma, V)\to M_D(\Sigma', V'=V\cap i(\Sigma'))$ is $D^V$-quasi-Poisson (where $D^{V\setminus V'}$ acts trivially on $M_D(\Sigma', V')$).
\end{thm}

Let now $(\mf d,\h,\g,\h^*)$ be a Manin quadruple, $D$ a connected Lie group integrating $\mf d$, and $H,G,H^*\subset D$ the connected subgroups integrating $\h,\g,\h^*$. We shall suppose that the map
$$H\times G\times H^*\to D,\quad (h,g,h^*)\mapsto hgh^*$$
is a diffeomorphism. (If we drop this assumption then we need to replace the moduli spaces considered below with appropriate locally diffeomorphic manifolds.)

\begin{example}[$\g$-quasi-Poisson Lie groups as moduli spaces]
Let us consider the moduli space $M_D(\triangle, V)$ for the triangle (i.e. for the disk with 3 marked points on the boundary). It is $D^3$-quasi-Poisson; after we take its  quotient by 
$$P_-\times P_-\times H\subset D\times D\times D,$$
 we get a $G$-quasi-Poisson manifold. For every marked point we specify on the picture the subgroup of $D$ by which we mod out:
$$
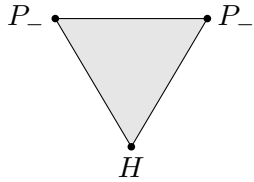
\begin{tikzpicture}[baseline=-1cm]
\coordinate[label=left:{$P_-$}] (A) at (0,0) ;
\coordinate[label=right:{$P_-$}] (B) at (2,0) ;
\coordinate[label=below:{$H$}] (C) at (1,-1.7);
\filldraw[fill=white!90!black] (A)--(B)--(C)--cycle;
\fill (A) circle (0.05) (B) circle (0.05) (C) circle (0.05);
\end{tikzpicture}
$$

By Theorem \ref{thm:moduli}, this $\g$-quasi-Poisson manifold $M_D(\triangle, V)/P_-\times P_-\times H$ can be obtained by a fusion followed by a reduction
$$
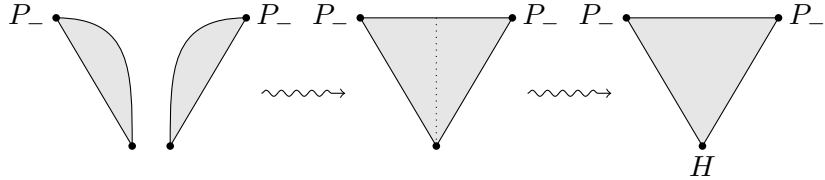
\begin{tikzpicture}[baseline=-1cm]
\coordinate[label=left:{$P_-$}] (A0) at (-4,0) ;
\coordinate[label=right:{$P_-$}] (B0) at (-1.5,0) ;
\coordinate (C0) at (-3,-1.7);
\coordinate (C0') at (-2.5,-1.7);
\filldraw[fill=white!90!black] (A0) .. controls (-3,0) and (-3,-.7) .. (C0)--cycle;
\filldraw[fill=white!90!black] (B0) .. controls (-2.5,0) and (-2.5,-.7) ..(C0')--cycle;
\fill (A0) circle (0.05) (B0) circle (0.05) (C0) circle (0.05) (C0') circle (0.05);
\draw[->,decorate,
     decoration={snake,amplitude=.4mm,segment length=2mm,post length=1mm}] (-1.3,-1) -- (-.2,-1);

\coordinate[label=left:{$P_-$}] (A) at (0,0) ;
\coordinate[label=right:{$P_-$}] (B) at (2,0) ;
\coordinate (C) at (1,-1.7);
\filldraw[fill=white!90!black] (A)--(B)--(C)--cycle;
\fill (A) circle (0.05) (B) circle (0.05) (C) circle (0.05);
\path[dotted,-] (C) edge (1,0);
\draw[->,decorate,
     decoration={snake,amplitude=.4mm,segment length=2mm,post length=1mm}] (2.2,-1) -- (3.3,-1);

\coordinate[label=left:{$P_-$}] (A1) at (3.5,0) ;
\coordinate[label=right:{$P_-$}] (B1) at (5.5,0) ;
\coordinate[label=below:{$H$}] (C1) at (4.5,-1.7);
\filldraw[fill=white!90!black] (A1)--(B1)--(C1)--cycle;
\fill (A1) circle (0.05) (B1) circle (0.05) (C1) circle (0.05);
\end{tikzpicture}
$$
$$M_D(\triangle, V)/P_-\times P_-\times H\cong((D/P_-)\circledast_{\mf d}(D/P_-))/H.$$
Here $D$, the moduli space for a bigon (which is $D^2$-quasi-Poisson), has $\pi=0$ (see  
Theorem \ref{thm:moduli}), and thus the $D$-quasi-Poisson space $D/P_-$ has also $\pi=0$. We can identify $D/P_-$ with $H$. We thus have an isomorphism
$$M_D(\triangle, V)/P_-\times P_-\times H\cong(H\circledast_{\mf d}H)/H$$
and thus, as we noticed in the proof of Theorem \ref{thm:construction-of-H} (where we denoted the $\mf d$-quasi-Poisson manifold $H$ with $\pi=0$ by $\hat H$), we have an isomorphism of $\g$-quasi-Poisson manifolds
$$M_D(\triangle, V)/P_-\times P_-\times H\cong H$$
where $H$ on the RHS is the $\g$-quasi-Poisson Lie group given by the Manin quadruple
$(\mf d,\h,\g,\h^*)$. Similarly, the moduli space given by 
$$
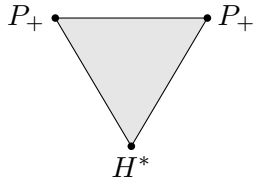
\begin{tikzpicture}[baseline=-1cm]
\coordinate[label=left:{$P_+$}] (A) at (0,0) ;
\coordinate[label=right:{$P_+$}] (B) at (2,0) ;
\coordinate[label=below:{$H^*$}] (C) at (1,-1.7);
\filldraw[fill=white!90!black] (A)--(B)--(C)--cycle;
\fill (A) circle (0.05) (B) circle (0.05) (C) circle (0.05);
\end{tikzpicture}
$$
is isomorphic, as a $\g$-quasi-Poisson manifold, to the dual $\g$-quasi-Poisson Lie group 
\begin{equation}\label{moduli-Hstar}
M_D(\triangle, V)/P_+\times P_+\times H^*\cong H^*.
\end{equation}
\end{example}

\begin{example}[Moment map from 3 marked points]
Let $\Sigma$ be a connected compact oriented surface and $V_\Sigma\subset\partial\Sigma$ a set of 3 marked points. Let $i:\triangle\to\Sigma$ be an embedding of a triangle sending the vertices $V_\triangle$ of $\triangle$ to $V_\Sigma$. Then, by Theorem \ref{thm:moduli}, the induced map
$$i^*:M_D(\Sigma,V_\Sigma)\to M_D(\triangle,V_\triangle)$$
is $D^3$-quasi-Poisson, and thus, in view of \eqref{moduli-Hstar}, it induces a $\g$-quasi-Poisson map
$$M_D(\Sigma,V_\Sigma)/P_+\times P_+\times H^*\to H^*,$$
i.e.\ a moment map.
$$
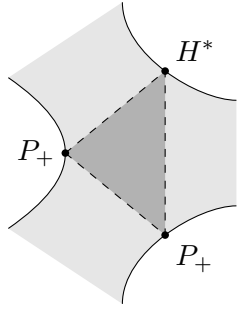
\begin{tikzpicture}
\fill[fill=white!90!black](0,1).. controls (1,0.3) and (1,-0.3)..  (0,-1)--
      (1.5,-2).. controls (1.5,-1.3) and (2.5,-0.7)..  (3,-0.7)--
      (3,0.7).. controls (2.5,0.7) and (1.5,1.3)..  (1.5,2) -- cycle;
\draw (0,1).. controls (1,0.3) and (1,-0.3).. coordinate(A) (0,-1)
      (1.5,-2).. controls (1.5,-1.3) and (2.5,-0.7).. coordinate(B) (3,-0.7)
      (1.5,2).. controls (1.5,1.3) and (2.5,0.7).. coordinate(C) (3,0.7);
\node at (A) [left] {$P_+$};
\node at (B) [below right] {$P_+$};
\node at (C) [above right] {$H^*$};
\filldraw[dashed, fill=white!70!black] (A)--(B)--(C)--cycle;
\fill (A) circle (0.05) (B) circle (0.05) (C) circle (0.05);
\end{tikzpicture}
$$
\end{example}

\begin{example}[Moment map from 2 marked points]\label{ex:2marked}

Let now $\Sigma$ be a connected compact oriented surface and $V_\Sigma=\{A,B\}\subset\partial\Sigma$ a set of 2 marked points. Let $i:\triangle\to\Sigma$ be a map which sends one vertex of $\triangle$ to $A$, the two remaining vertices to $B$, and is an embedding of $\triangle$ with two vertices identified to $\Sigma$. By Theorem \ref{thm:moduli} the map
$$i^*:M_D(\Sigma,V_\Sigma)\to M_D(\triangle, V_\triangle)^\circledast$$
is $D^2$-quasi-Poisson, where $M_D(\triangle, V_\triangle)^\circledast$ is $M_D(\triangle, V_\triangle)$ with fusion applied at the two identified vertices. (Geometrically, $M_D(\triangle, V_\triangle)^\circledast$ is the moduli space for an annulus with two marked points; the two points can be either on the same boundary circle, or on both of them, depending on the order of the fusion, giving two different quasi-Poisson structures.)

The projection
$$p:M_D(\triangle, V_\triangle)^\circledast/P_+\times H^*\to M_D(\triangle, V_\triangle)/P_+\times P_+\times H^*$$
is $\g$-quasi-Poisson. As a result,
$$p\circ i^*:M_D(\Sigma,V_\Sigma)/P_+\times H^*\to M_D(\triangle, V_\triangle)/P_+\times P_+\times H^*\cong H^*$$
is a $\g$-quasi-Poisson map to the dual $\g$-quasi-Poisson group $H^*$, i.e.\ it is a moment map.

As a simple example, let $\Sigma$ be an annulus and let $V_\Sigma$ consist of two points, one on each boundary circle of $\Sigma$.
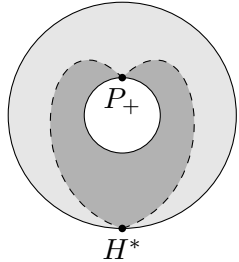
$$
\begin{tikzpicture}
\filldraw[fill=white!90!black] (0,0)circle [radius=1.5cm];
\coordinate (A) at (0,0.5) ;
\coordinate[label=below:{$H^*$}] (B) at (0,-1.5) ;
\filldraw[fill=white!70!black,dashed] (B)..controls (1.5,-1) and (1,1.5) ..(A)..controls (-1,1.5) and (-1.5,-1) ..(B);
\filldraw[fill=white] (0,0)circle [radius=0.5cm];
\node at (A) [below] {$P_+$};
\fill (A) circle (0.05) (B) circle (0.05);
\end{tikzpicture}
$$
In this case $M_D(\Sigma,V_\Sigma)\cong D\times D$, with the identification given by the holonomies along the dashed curves. Moreover $M_D(\Sigma,V_\Sigma)/P_+\times H^*\cong D$, as the projection
$$D\times D\cong M_D(\Sigma,V_\Sigma) \to M_D(\Sigma,V_\Sigma)/P_+\times H^*$$
restricts to a diffeomophism on $\{1\}\times D\subset D\times D$.

The $\g$-quasi-Poisson structure on $D\cong M_D(\Sigma,V_\Sigma)/P_+\times H^*$ is
$$
  \pi=\frac{1}{2}\left[(E^i\wedge E_i)^L-(E^i \wedge E_i)^R-(e^\alpha)^L\wedge (e_\alpha)^R\right]
$$
with $\g$ acting on $D$ by conjugation.
One can check that it makes $D$ to a $\g$-quasi-Poisson group and that $H\subset D$ (though not $H^*\subset D$) is a $\g$-quasi-Poisson inclusion. The moment map $D\to H^*$
is the projection $D\to D/P_+\cong H^*$. 
\end{example}

\section{Quantization to braided Hopf algebras}

One of the main reason for introducing $\g$-quasi-Poisson Lie groups and Lie $\g$-quasi-bialgebras is that they can be quantized to braided Hopf algebras. We shall construct this quantization procedure in this section.

Let us briefly describe Drinfeld's construction of braided monoidal categories (BMCs) using a Drinfeld associator \cite{D}.
If $U$ and $V$ are two $U\g$-modules, let $t^{U,V}:U\otimes V\to U\otimes V$ be given by $t^{U,V}:=(\rho_U\otimes\rho_V)(t)$ (where, as above, $t\in(S^2\g)^\g\subset\g\otimes\g$), and let $\sigma^{U,V}:U\otimes V\to V\otimes U$ be the flip. Let $U\g\text{-mod}_\hbar$ be the category of $U\g$-modules, where the morphisms are allowed to be formal power series in $\hbar$. If $\Phi\in K\la\la x,y\ra\ra$ is a Drinfeld associator, where $K$ is our base field of characteristic 0, then the braiding
$$\beta^{U,V}=\sigma^{U,V}\circ\exp\Bigl(\frac{\hbar}{2}\,t^{U,V}\Bigr):U\otimes V\to V\otimes U$$
and the associativity constraint
$$\gamma^{U,V,W}=\gamma_0^{U,V,W}\circ\Phi(\hbar\,t^{U,V},\hbar\,t^{V,W}):(U\otimes V)\otimes W\to U\otimes(V\otimes W),$$
where 
$$\gamma_0^{U,V,W}:(U\otimes V)\otimes W\to U\otimes(V\otimes W)$$
is the standard associativity constraint, make $U\g\text{-mod}_\hbar$ to a BMC. We shall denote this braided monoidal category by $U\g\text{-mod}_\hbar^\Phi$.

If $M$ is a $\g$-quasi-Poisson manifold then its \emph{deformation quantization} is a 
 $\g$-equivariant bilinear product $*$ on $C^\infty(M)$
$$f_1*f_2=f_1f_2+\hbar\,B_1(f_1,f_2)+\hbar^2 B_2(f_1,f_2)+\dots$$
 where $B_k$ are bidifferential operators, such that $(C^\infty(M),*,1)$ is an associative algebra in the BMC $U\g\text{-mod}_\hbar^\Phi$, and such that
$$\{f_1,f_2\}=B_1(f_1,f_2)-B_1(f_2,f_1).$$
It is not expected to exist in general. However, if $H$ is a $\g$-quasi-Poisson Lie group, such a $*$ does exist, and, moreover, is a part of a Hopf algebra structure. In order to describe this quantization it is convenient to introduce the following category (a version of $U\g\on{-mod}_\hbar$).

\begin{defn}
The category \emph{$\g\text{-man}_\hbar$} has $\g$-manifolds as objects, and a morphism $F:M\to N$ is a $\g$-equivariant linear map
$$F^*:C^\infty(N)\to C^\infty(M)[[\hbar]],\quad F^*=F_0^*+\hbar\,F_1^*+\hbar^2F_2^*+\dots$$
such that $F_0^*=f^*$ for some smooth map $f:M\to N$ and such that 
 $F_k^*$, $k\geq 1$, are differential operators w.r.t.\ the algebra map $F_0$. The composition of morphisms is given by $(F\circ G)^*:=G^*\circ F^*$.
\end{defn}

The category $\g\text{-man}_\hbar$ is symmetric monoidal, with the product $M_1\otimes M_2:=M_1\times M_2$. Its deformation to a BMC via a Drinfeld associator $\Phi\in\R\la\la x,y\ra\ra$ will be denoted by $\g\text{-man}_\hbar^\Phi$.

Let $\g\on{-man}$ denote the symmetric monoidal category of $\g$-manifolds.
Whenever $M$ is a $\g$-manifold then $(M,\Delta,\epsilon)$ is a comonoid in $\g\text{-man}$, where $\Delta:M\to M\times M$ is the diagonal map, and  $\epsilon:M\to\text{point}$ is the unique map.
A comonoid structure $(M,\Delta_\hbar,\epsilon)$  in $\g\text{-man}_\hbar^\Phi$, such that $(\Delta_\hbar^*)_0=\Delta^*$, is equivalent to a $*$-product on $M$ via
$$f*g=\Delta_\hbar^*(f\otimes g)$$
making $C^\infty(M)$ to an associative algebra in $U\g\on{-mod}_\hbar^\Phi$.

We can now formulate the main result of this section.
\begin{thm}\label{thm:quant}
If $H$ is a $\g$-quasi-Poisson Lie group then there is a Hopf algebra structure on $H$ in the BMC {$\g\on{-man}_\hbar^\Phi$} deforming the Hopf algebra (i.e.\ Lie group) structure of $H$ in {$\g\on{-man}$}, such that the deformed coproduct is a deformation quantization of the $\g$-quasi-Poisson manifold $H$.

Similarly, if $(\h,\dot\rho,\delta)$ is a Lie $\g$-quasi-bialgebra, there is a deformation of the Hopf algebra structure on $U\h$   to a Hopf algebra structure in the BMC {$U\g\on{-mod}_\hbar^\Phi$}, such that, for $v\in\g\subset U\g$, $\Delta(v)-\Delta^{op}(v)=\hbar\,\delta(v)+O(\hbar^2)$.
\end{thm}

The proof of the theorem is an application of the method of quantization of Lie bialgebras presented in \cite{S} (one can probably also use the original approach of  Etingof and Kazhdan \cite{EK} in combination with the $\g$-quasi-Poisson group $D$ of Example \ref{ex:2marked}, but it seems more complicated). Let us first observe that if $(Q,\Delta_Q,\epsilon_Q)$ and $(Q',\Delta_{Q'},\epsilon_{Q'})$ are  comonoids in a BMC $\mathcal D$ then $Q\otimes Q'$ is a comonoid as well, with the counit $\epsilon_Q\otimes\epsilon_{Q'}$ and the coproduct
$$
\begin{tikzpicture}[baseline=-1cm]
\coordinate (diff) at (0.7,0);
\coordinate (dy) at (0,-0.8);
\node(A1) at (0,0) {$(Q$};
\node(B1) at ($(A1)+(diff)$) {$Q')$};
\node(A2) at (2,0) {$(Q$};
\node(B2) at ($(A2)+(diff)$) {$Q')$};
\node(A3) at ($(A1)+(0,-2.5)+0.5*(diff)$) {$Q$};
\node(B3) at ($(A2)+(0,-2.5)+0.5*(diff)$) {$Q'$};
\coordinate(A) at ($(A3)-(dy)$);
\coordinate(B) at ($(B3)-(dy)$);
\draw[line width=1ex,white] (B1)..controls +(0,-1) and +(0,1)..(B);
\draw (B1)..controls +(0,-1) and +(0,1)..(B);
\draw[line width=1ex,white]  (A2)..controls +(0,-1) and +(0,1)..(A);
\draw (A2)..controls +(0,-1) and +(0,1)..(A);
\draw (A1)..controls +(0,-1) and +(0,1)..(A);
\draw (B2)..controls +(0,-1) and +(0,1)..(B);
\draw (B)--(B3);
\draw(A)--(A3);
\end{tikzpicture}
$$
\begin{thm}[\cite{S}]\label{thm:Hopf}
Let $\mathcal D$ and $\mathcal C$ be BMCs, $(Q,\Delta_Q,\epsilon_Q)$ a cocommutative comonoid in $\mathcal D$, and $F:\mathcal D\to\mathcal C$ a braided colax monoidal functor with these invertibility properties:
the composition
$$F(Q)\xrightarrow{F(\epsilon_Q)}F(1_\mathcal D)\to1_\mathcal C$$
is an isomorphism, and  for every objects $X,Y\in\mathcal D$ the morphism
$$\tau^{(Q)}_{X,Y}:F((X\otimes Q)\otimes Y)\to F(X\otimes Q)\otimes F(Q\otimes Y),$$
defined as the composition
\begin{multline*}
F((X\otimes Q)\otimes Y)\xrightarrow{F((\on{id}_X\otimes\Delta_Q)\otimes\on{id}_Y)}
F\bigl((X\otimes (Q\otimes Q))\otimes Y\bigr)\cong \\
\cong F\bigl((X\otimes Q)\otimes(Q\otimes Y)\bigr)\to F(X\otimes Q)\otimes F(Q\otimes Y),
\end{multline*}
is an isomorphism.

Then $F(Q\otimes Q)$ is a Hopf algebra in $\mathcal C$, with the structure defined as follows:
\begin{itemize}
\item The coalgebra structure on $F(Q\otimes Q)$ is inherited from the coalgebra structure  on $Q\otimes Q$.
\item The product on $F(Q\otimes Q)$ is the composition
\begin{multline*}
F(Q\otimes Q)\otimes F(Q\otimes Q)\xrightarrow{{\tau^{(Q)}_{Q,Q}}^{-1}}F((Q\otimes Q)\otimes Q)\\ \xrightarrow{F(\on{id}_Q\otimes\epsilon_Q\otimes\on{id}_Q)}F(Q\otimes Q).
\end{multline*}
\item The unit is
$$
1_{\mathcal C}\cong F(Q)\xrightarrow{F(\Delta_Q)}F(Q\otimes Q).
$$
\item The antipode is 
$$F(Q\otimes Q)\xrightarrow{F(\beta^{Q,Q})}F(Q\otimes Q)$$
where $\beta^{Q,Q}:Q\otimes Q\to Q\otimes Q$ is the braiding in $\mathcal D$.
\end{itemize}
\end{thm}

\begin{proof}[Proof of Theorem \ref{thm:quant}]
Let $(\mf d,\h,\g,\h^*)$ be the Manin quadruple corresponding to the Lie $\g$-quasi-bialgebra $\h$. Let $\mathcal C=\g\on{-man}_\hbar^\Phi$ and let $\mathcal D$ be the full subcategory of $\mf d\on{-man}_\hbar^\Phi$ consisting of those $\mf d$-manifolds for which the action of $\h\subset\mf d$ integrates to a free and proper action of the group $H$. Let $F:\mathcal D\to\mathcal C$ be given by
$$F(M)=M/H$$
and let $Q=H$ with the action $\hat\rho$ of $\mf d$ (see Proposition \ref{prop:hat-rho}) and with the undeformed coproduct $\Delta_Q$ and undeformed counit $\epsilon_Q$.

By \cite[Propositions 1 and 2]{LBS3} the functor $F$, with the coherence maps $(M\times N)/H\to M/H\times N/H$ given by the natural projections, is a braided colax monoidal functor and $(Q,\Delta_Q,\epsilon_Q)$ is a cocommutative comonoid in $\mathcal D$. The hypotheses of Theorem \ref{thm:Hopf} are satisfied, hence 
$$F(Q\otimes Q)=(H\times H)/H\cong H$$
is a Hopf algebra in $\g\on{-man}_\hbar^\Phi$. The identification $(H\times H)/H\cong H$ we use is $(h_1,h_2)\mapsto h_1 h_2^{-1}$, as in the proof of Theorem \ref{thm:construction-of-H}.

This Hopf algebra structure (in $\g\on{-man}_\hbar^\Phi$) on $H$ is clearly a deformation of the Lie group structure on $H$. To finish the proof we need to check that the $*$-product on $H$ we obtained is indeed a deformation quantization of the $\g$-quasi-Poisson structure on $H$. It follows from the fact \cite{LBS3} that if $(M^{(1)},\Delta^{(1)}_\hbar)$ and $(M^{(2)},\Delta^{(2)}_\hbar)$ are deformation quantizations of two $\mf d$-quasi-Poisson structures on manifolds $M^{(1)}$ and $M^{(2)}$ then their tensor product in $\mf d\on{-man}_\hbar^\Phi$ is a deformation quantization of the fusion $M^{(1)}\circledast M^{(2)}$. As a result the comonoid $Q\otimes Q$ is a deformation quantization of $\hat H\circledast_{\mf d}\hat H$, where $\hat H$ denotes, as in the proof of Theorem \ref{thm:construction-of-H}, $H$ with the action $\hat\rho$ and with $\pi=0$, and thus $F(Q\otimes Q)$ is a quantization of the $\g$-quasi-Poisson manifold $(\hat H\circledast_{\mf d}\hat H)/H$, which is $H$ with its $\g$-quasi-Poisson structure.

The deformation of the Hopf algebra structure on $U\h$ is constructed similarly: we set $\mathcal D=U\mf d\on{-mod}_\hbar^\Phi$, $\mathcal C=U\mf g\on{-mod}_\hbar^\Phi$, $F(V)=V/(\h\cdot V)$, and $Q=U\h=U\mf d/(\mf p_- U\mf d)$ with its original cocommutative coalgebra structure. Finally we identify $F(U\h\otimes U\h)$ with $U\h$ via $x\otimes y\mapsto x\,S_0(y)$, where $S_0:U\h\to U\h$ is the original (non-deformed) antipode. The coproduct in the coalgebra $U\h\otimes U\h$ (in the category $\mathcal D$) applied to $x\otimes 1$, $x\in\h$, is
$$\Delta(x\otimes1)=x\otimes1\otimes1\otimes1+1\otimes 1\otimes x\otimes1-\frac{\hbar}{2}\,1\otimes\delta(x)\otimes1+O(\hbar^2)\in U\h^{\otimes 4}$$
and thus the coproduct in $U\h\cong F(U\h\otimes U\h)$ applied to $x\in\h$ is
$$\Delta(x) =x\otimes 1+1\otimes x+\frac{\hbar}{2}\,\delta(x)+O(\hbar^2),$$
so indeed
$$(\Delta-\Delta^{op})(x)=\hbar\,\delta(x)+O(\hbar^2)$$
as we wanted to show.
\end{proof}

\end{document}